\documentclass[a4paper]{amsart}

\usepackage{amsmath, amssymb, latexsym, amsfonts}
\usepackage[matrix,arrow,curve]{xy}

\numberwithin{equation}{section}
\theoremstyle{plain}
\newtheorem{theorem}[equation]{Theorem}
\newtheorem{corollary}[equation]{Corollary}
\newtheorem{lemma}[equation]{Lemma}
\theoremstyle{definition}
\newtheorem{definition}[equation]{Definition}
\newtheorem{notation}[equation]{Notation}
\theoremstyle{remark}
\newtheorem{remark}[equation]{Remark}

\newcommand{\SL}{{\mathrm{SL}}}

\newcommand{\Proj}{\mathbb{P}}
\newcommand{\Hom}{{\mathrm{Hom}}}
\newcommand{\iHom}{{\underline{\mathrm{Hom}}}}
\newcommand{\Z}{{\mathbb{Z}}}

\newcommand{\GW}{{\mathrm{GW}}}

\newcommand{\K}{\mathrm{K}}
\newcommand{\MW}{\mathrm{MW}}

\newcommand{\A}{\mathbb{A}}
\newcommand{\SSp}{\mathbb{S}}
\newcommand{\PP}{\mathbb{P}}
\newcommand{\rank}{\operatorname{rank}}
\newcommand{\Th}{\operatorname{Th}}

\newcommand{\Cone}{\operatorname{Cone}}

\newcommand{\colim}{\operatorname*{colim}}
\newcommand{\triv}{\mathbf{1}}
\newcommand{\struct}{\mathcal{O}}
\newcommand{\Gm}{{\mathbb{G}_m}}
\newcommand{\GL}{\mathrm{GL}}
\newcommand{\SH}{\mathcal{SH}}

\newcommand{\Ht}{\mathrm{H}}
\newcommand{\T}{\mathrm{T}}

\newcommand{\Spec}{\operatorname{Spec}}

\newcommand{\coker}{\operatorname{coker}}
\newcommand{\coh}{\mathrm{H}}
\newcommand{\Sph}{\mathrm{S}}
\newcommand{\Zar}{\mathrm{Zar}}

\newcommand{\Sm}{\mathrm{Sm}}
\newcommand{\Fr}{\mathrm{Fr}}
\newcommand{\ZF}{\mathbb{Z}\mathrm{F}}
\newcommand{\Nis}{\mathrm{Nis}}
\newcommand{\heart}{\Pi_*(k)}
\newcommand{\RS}{\mathcal{RS}_\bullet}

\begin{document}

\title{On the zeroth stable $\A^1$-homotopy group of a smooth curve}

\author{Alexey Ananyevskiy}
\email{alseang@gmail.com}
\address{
	St. Petersburg Department of Steklov Institute of Mathematics, Russian Academy of Sciences,\\
	Fontanka 27, St. Petersburg 191023 Russia \\
	Chebyshev Laboratory, St. Petersburg State University, 14th Line V.O. 29B, Saint Petersburg 199178 Russia}

\begin{abstract}
	We provide a cohomological interpretation of the zeroth stable $\A^1$-homotopy group of a smooth curve over an infinite perfect field. We show that this group is isomorphic to the first Nisnevich (or Zariski) cohomology group of a certain sheaf closely related to the first Milnor--Witt $\K$-theory sheaf. This cohomology group can be computed using an explicit Gersten-type complex. We show that if the base field is algebraically closed then the zeroth stable $\A^1$-homotopy group of a smooth curve coincides with the zeroth Suslin homology group that was identified by Suslin and Voevodsky with a relative Picard group. As a consequence we reobtain a version of Suslin's rigidity theorem.
\end{abstract}

\maketitle

\section{Introduction}

A.~Suslin and V.~Voevodsky computed singular homology (also known as Suslin homology) groups of a smooth relative curve admitting a good compactification. For a curve over a field the result for the zeroth Suslin homology group reads as follows.

\begin{theorem}[{\cite[Theorem~3.1]{SV96}}] \label{thm:suslin-voevodsky}
Let $C_o$ be a smooth curve over a field $k$, let $C$ be its smooth compactification and $D=C-C_o$. Then there exists a natural isomorphism
\[
\coh^S_0(C_o)\cong \mathrm{Pic}(C,D).
\]
\end{theorem}
\noindent
The goal of this paper is to give a similar description of the zeroth stable motivic homotopy group $\pi^{\A^1}_0(C_o)$.

The relative Picard group $\mathrm{Pic}(C,D)$ from Theorem~\ref{thm:suslin-voevodsky} consists of the isomorphism classes of pairs $(L,\phi)$ with $L$ being a line bundle over $C$ and $\phi\colon L|_{D} \xrightarrow{\simeq} \triv_{D}$ being a trivialization of $L$ over $D$. The group structure is given by the tensor product. As in \cite[Lemma~2.1]{SV96}, one can identify the relative Picard group with sheaf cohomology groups,
\[
\mathrm{Pic}(C,D) \cong \coh^1_{\Zar}(C,(\Gm)_{C,D}) \cong \coh^1_{\Nis}(C,(\Gm)_{C,D}) \cong \coh^1_{\mathrm{\acute{e}t}}(C,(\Gm)_{C,D}),
\]
where $(\Gm)_{C,D}=\ker \left[ (\Gm)_C \to (i_D)_* (\Gm)_D \right]$ with $i_D\colon D\to C$ being the closed embedding. In other words, $(\Gm)_{C,D}$ is the sheaf of invertible functions on $C$ trivial (equal to $1$) over $D$. Moreover, one has the following long exact sequence \cite[Lemma~2.3]{SV96}.
\[
0\to (\Gm)_{C,D}(C) \to \{f\in k(C)\,|\, \text{$f$ is defined and equal to $1$ over $D$} \} \to \bigoplus_{x\in C_o^{(1)}} \Z \to \mathrm{Pic}(C,D) \to 0.
\]
One can regard the two terms in the middle as a Gerten-type resolution for $(\Gm)_{C,D}$ and the exact sequence yields that the first sheaf cohomology group of $(\Gm)_{C,D}$ can be computed as the first cohomology group of the corresponding Gersten-type complex. In \cite{SV96} the authors show that  $\coh^S_0(C_o)$ also sits in the above exact sequence, thus  $\coh^S_0(C_o)\cong \mathrm{Pic}(C,D)$. From our point of view the essential part of the theorem is that the Suslin homology group $\coh^S_0(C_o)$ is isomorphic to the sheaf cohomology group of $(\Gm)_{C,D}$ (and sits in the above exact sequence), while $\mathrm{Pic}(C,D)$ gives a nice geometrical interpretation of the latter group.

The zeroth Suslin homology group of a curve can be defined as the group of morphisms in Voevodsky's triangulated motivic category $\mathrm{DM}(k)$ \cite{MVW06} from the motive of the base field to the motive of the curve,
\[
\coh^S_0(C_o) = \Hom_{\mathrm{DM}(k)}(M(\Spec k), M(C_o)).
\]
It looks reasonable to address the similar problem for the zeroth stable motivic homotopy group, i.e. to compute
\[
\pi^{\A^1}_0(C_o) = \Hom_{\SH(k)}(\Sigma^\infty_\T (\Spec k)_+, \Sigma^\infty_\T (C_o)_+),
\]
where $\SH(k)$ is the stable motivic homotopy category \cite{MV99,V98,Mor04}. A.~Asok and C.~Haesemeyer described this group for a smooth projective variety {\cite[Theorem~4.3.1]{AH11}}; see also Theorem~\ref{thm:hom_variety} of the present paper for a slight variation of the result: for the description of
\[
\Hom_{\SH(k)}(\Sigma^\infty_\T (\Spec k)_+, \Sigma^l_{\Gm}\Sigma^\infty_\T X_+)
\]
with $X$ being smooth and projective. For a curve the result looks as follows.
\begin{theorem}[{\cite[Theorem~4.3.1]{AH11}} or Theorem~\ref{thm:hom_variety}] \label{thm:intr_AH}
Let $C$ be a smooth projective curve over an infinite perfect field $k$ of characteristic unequal to $2$. Then there exists a natural isomorphism
\[
\pi^{\A^1}_0(C) \cong \coh^1_\Nis(C,\K_1^{\MW}\otimes \omega_{C}).
\]
\end{theorem}
\noindent 
The unramified Milnor--Witt $\K$-theory sheaf $\K_n^{\MW}$ that appears in the theorem can be defined by means of the cartesian square 
\[
\xymatrix{
\K_n^{\MW} \ar[r] \ar[d] & I^n \ar[d] \\
\K_n^{\mathrm{M}} \ar[r] &  \K_n^{\mathrm{M}}/2\K_n^{\mathrm{M}} \cong I^n/I^{n+1},
}
\]
where $I=\ker\left[ W \xrightarrow{\rank} \Z/2\Z \right]$ is the sheaf associated to the fundamental ideal in the Witt ring and $\K_n^{\mathrm{M}}$ is the unramified Milnor $\K$-theory sheaf (see \cite[Chapter~3]{Mor12}). This description gives rise to an action of $\Gm$ on $\K_n^{\MW}$: $a\in \Gm(U)=k[U]^*$ acts trivially on $\K_n^{\mathrm{M}}(U)$ and via multiplication with the rank one quadratic form $\langle a \rangle$ on $I^n(U)$. The action allows one to twist $\K_n^{\MW}$ by a line bundle, see Definition~\ref{def:twist} for the details. In particular, one has the twisted Milnor--Witt $\K$-theory sheaf $\K_n^{\MW}\otimes \omega_{C}$ being the sheaf associated with the presheaf
\[
U \mapsto \K_n^{\MW}(U) \otimes_{\Z[k[U]^*]} \Z[\Gamma(U,\omega_{C})^0],
\]
where $\Gamma(U,\omega_{C})^0$ is the set of nowhere vanishing sections (i.e. trivializations) of the canonical bundle $\omega_{C}$ restricted to $U$.

The cohomology group from Theorem~\ref{thm:intr_AH} can also be computed using a two term Gersten complex (Rost-Schmid complex in the notation of \cite[Chapter~5]{Mor12}),
\[
(\K_1^{\MW}\otimes \omega_{C})(k(C)) \to \bigoplus_{x\in C^{(1)}} \GW(k(x)) \to \coh^1_\Nis(C,\K_1^{\MW}\otimes \omega_{C}) \to 0,
\]
see \cite[{Corollary~4.1.12}]{AH11} or Theorem~\ref{thm:Hom_coh} of the present paper.

A close comparison of the results of Suslin--Voevodsky and Asok--Haesemeyer combined with the isomorphisms $\Gm\cong \K_1^{\mathrm{M}}$, $\Z\cong \K_0^{\mathrm{M}}$ and $\GW\cong \K_0^{\MW}$ immediately suggests a general statement for $\pi^{\A^1}_0(C_o)$ that is proved in the present paper.

\begin{theorem}[{Theorem~\ref{thm:hom_curve} and Remark~\ref{rem:hom_curve}}] \label{thm:hom_curve_intr}
Let $C_o$ be a smooth curve over an infinite perfect field $k$, let $C$ be its smooth compactification and $D=C-C_o$. Then there exist natural isomorphisms
\[
\pi^{\A^1}_0(C_o)\cong \coh^1_{\Nis}(C,(\K_1^{\MW}\otimes \omega_{C})_{C,D})\cong \coh^1_{\Zar}(C,(\K_1^{\MW}\otimes \omega_{C})_{C,D}),
\]
where 
\[
(\K_1^{\MW}\otimes \omega_{C})_{C,D}= \ker\left[ \K_1^{\MW}\otimes \omega_{C} \to (i_D)_*(\K_1^{\MW}\otimes i_D^*\omega_{C}) \right]
\]
for the closed embedding $i_D\colon D\to C$. 

Moreover, there is an exact sequence
\[
\{a\in (\K^{\MW}_{1}\otimes \omega_C)(k(C))\,|\, \text{$a$ is defined and trivial over $D$}\} \to \bigoplus_{x\in C_o^{(1)}} \mathrm{GW} (k(x)) \xrightarrow{\Theta} \pi^{\A^1}_0(C_o) \to 0
\]
such that $\Theta(\langle 1 \rangle_y)=\Sigma^\infty_T (i_y)_+$ for a rational point $y\in C_o(k)$, the associated closed embedding $i_y\colon \Spec k \to C_o$ and $\langle 1 \rangle_y$ being the unit quadratic form belonging to the corresponding direct summand.
\end{theorem}
\noindent We also give a similar description for the group $\Hom_{\SH(k)}(\Sigma^\infty_\T (\Spec k)_+,\Sigma^l_{\Gm}\Sigma^\infty_\T (C_o)_+)$, see Theorem~\ref{thm:hom_curve}.

In the case of a projective curve $C_o=C$ we obtain precisely the description given by A.~Asok and C.~Haesemeyer. On the other hand, the right-hand sides in Theorems~\ref{thm:suslin-voevodsky} and~\ref{thm:hom_curve_intr} are the same up to the substitution of the Milnor--Witt $\K$-theory for the Milnor $\K$-theory. Indeed, one can easily show that $\K_1^{\mathrm{M}}\otimes \omega_{C}\cong \K_1^{\mathrm{M}}$, since the action of $\Gm$ on $\K_1^{\mathrm{M}}$ is trivial. In particular, we have the following corollary.

\begin{corollary}
Let $C_o$ be a smooth curve over an algebraically closed field $k$, let $C$ be its smooth compactification and $D=C-C_o$. Then there exist natural isomorphisms
\[
\pi^{\A^1}_0(C_o)\cong \coh^S_0(C_o) \cong \mathrm{Pic}(C,D).
\]
\end{corollary}
This corollary allows one to avoid the technical discussion of transfers in the proof of the following version of rigidity theorem ({cf. \cite[Theorem~1.9]{Ya04} and \cite[Theorem~4.9]{RO08}}), since the pairing $\mathrm{Pic}(C,D) \times A(C_o) \to A(\Spec k)$ is immediate. The claim follows from divisibility of the Picard group as in loc. cit.
\begin{theorem}
Let $C_o$ be a smooth curve over an algebraically closed field $k$ and let $i_0\colon \Spec k \to C_o$ and $i_1\colon \Spec k \to C_o$ be two closed points. Then for $A\in \SH(k)$ such that $nA=0$ for some integer $n$ coprime to $\operatorname{char} k$ one has
\[
i_0^*=i_1^*\colon A(C_o)\to A(\Spec k).
\]
\end{theorem}
Moreover, it is well-known that the above theorem yields the following one via an argument due to A.~Suslin \cite[the proof of Main~Theorem]{Sus83}.
\begin{theorem}[{cf. \cite[Theorem~1.10]{Ya04} and \cite[Theorem~4.10]{RO08}}]
Let $K/k$ be an extension of algebraically closed fields, let $X$ be a smooth variety over $k$ and let $A\in \SH(k)$ be such that $nA=0$ for some integer $n$ coprime to $\operatorname{char} k$. Then the map
\[
A(X) \to A(X\times_{\Spec k}\Spec K)
\]
is an isomorphism.
\end{theorem}

The paper is organized in the following way. The reader is assumed to be familiar with the basic notions of stable motivic homotopy theory (see \cite{V98,MV99,Mor04}). In Section~2 we recall some properties of stable motivic homotopy sheaves and study the orientation phenomena, i.e. relate the values of stable motivic homotopy sheaves on Thom spaces of various vector bundles. In Section~3 we organize well-known facts about the coniveau spectral sequence and show that the Rost-Schmid complex computes Nisnevich and Zariski cohomology of stable motivic homotopy sheaves. In Section~4 we derive a variation of the Rost-Schmid complex for a pair $(C,D)$ with $C$ being a smooth curve and $D$ a finite collection of closed points. In Section~5 we adopt the approach of Asok--Haesemeyer applying Atiyah duality in order to obtain a cohomological description of the zeroth stable homotopy group. Then we apply the developed technique deriving an explicit description for the zeroth stable motivic homotopy group of $\Gm$-suspensions of a smooth projective variety and of a smooth curve. In the Appendix we provide some details on the theory of presheaves with framed transfers (see \cite{GP14,GP15}) that are used in the present paper.

\begin{notation}
We adopt the following notation and conventions.
\begin{itemize}
	\item
	$k$ is an infinite perfect field.
	\item
	$\Sm_k$ is the category of smooth varieties over $k$.
	\item
	$\Sm_k/X$ is the category of regular morphisms $Y\to X$ with $Y$ and $X$ being smooth varieties, i.e. $\Sm_k/X$ is the slice category over $X$. Note that this category is different from the category $\Sm_X$ of smooth schemes over $X$.
	\item
	$\SH(k)$ is the stable motivic homotopy category over $k$ (see \cite{V98,MV99,Mor04}).
	\item
	$[A,B]=\Hom_{\SH(k)}(A,B)$ for $A,B\in \SH(k)$.
	\item
	$\T=\A^1/(\A^1-0)$, $\Gm=(\A^{1}-0,1)$. 
	\item
	$\Th(E)=E/(E-z(X))$ is the Thom space of a vector bundle $E$ over a smooth variety $X$ with the zero section $z\colon X\to E$.
	\item
	$A(l)=\Sigma^l_{\Gm} A$ and $A[i]=\Sigma^i_{\Sph^1} A$ for $A\in \SH(k)$.
	\item
	$[X,A]=[\Sigma_\T^\infty X_+,A]$ for $X\in \Sm_k$ and $A\in\SH(k)$.
	\item
	$[P,A]=[\Sigma_\T^\infty P,A]$ for a pointed Nisnevich sheaf $P$ on $\Sm_k$ and $A\in\SH(k)$. In particular, we write
	\[
	[X/Y,A]=[\Sigma_\T^\infty (X/Y),A]
	\]
	for smooth varieties $Y\subset X$. Here the varieties are treated as representable Nisnevich sheaves via the Yoneda embedding.
	\item
	An essentially smooth scheme is a noetherian scheme which is the inverse limit of a filtering system with each transition morphism being an etale affine morphism of smooth varieties. For a presheaf on the category of smooth varieties we extend it to the category of essentially smooth schemes by the respective colimit.
\end{itemize}
\end{notation}

\section{Homotopy sheaves}
In this section we recall the notions of homotopy sheaves and presheaves of spectra and study the orientation phenomena for these (pre-)sheaves relating its values on Thom spaces of vector bundles.

\begin{definition}
For $A\in \SH(k)$ and $i,n\in \Z$ let $\pi_{i}^{\A^1}(A)_n(-)$ be the \textit{homotopy presheaf} of $A$ that is given by
\[
\pi_{i}^{\A^1}(A)_n(U)=[U[i],A(n)]
\]
for $U\in \Sm_k$. The associated Zariski sheaf is denoted $\underline{\pi}_{i}^{\A^1}(A)_n$ and referred to as a \textit{homotopy sheaf} of $A$. Note that for a smooth variety $X$ the isomorphism $\T\cong \Sph^1\wedge \Gm$ \cite[Lemma~2.15]{MV99} identifies
\[
[ \Sigma^n_\T X_+,A]\cong \pi_{n}^{\A^1}(A)_{-n}(X).
\]
\end{definition}

\begin{remark}
It follows from the general theory of presheaves with framed transfers developed in \cite{GP15} that $\underline{\pi}_{i}^{\A^1}(A)_n$ are Nisnevich sheaves (see Definition~\ref{def:homotopy_transfers} and Lemma~\ref{lm:ZarNis} in the current paper). Thus the given above definition of $\underline{\pi}_{i}^{\A^1}(A)_n$ coincides with the usual one, where Nisnevich sheafification is considered.
\end{remark}

\begin{definition}
Let $E$ be a rank $r$ vector bundle over a smooth variety $X$. For $A\in \SH(k)$ let $\pi_{i}^{\A^1}(A)_n(-;E)$ be the presheaf on $\Sm_k/X$ given by
\[
\pi_{i}^{\A^1}(A)_n(Y;E) = [\Th(f^*E)[i-r],A(n+r)]
\]
with $f\colon Y\to X$ being the structure map. We denote $\underline{\pi}_{i}^{\A^1}(A)_n(-;E)$ the associated Zariski sheaf on $\Sm_k/X$.
\end{definition}

\begin{remark}
The shift of the indices in the above definition guarantees that a trivialization $E\cong \triv_X^r$ induces an isomorphism of presheaves on $\Sm_k/X$
\[
\pi_{i}^{\A^1}(A)_n(-;E)\cong \pi_{i}^{\A^1}(A)_n(-).
\]
\end{remark}

\begin{definition} \label{def:Gm_action}
Let $E$ be a rank $r$ vector bundle over a smooth variety $X$ and $A\in\SH(k)$. For a regular morphism of smooth varieties $f\colon Y\to X$ we have a right action of the group of vector bundle automorphisms $\GL(f^*E)$ on $\pi_{i}^{\A^1}(A)_n(Y;E)=[\Th(f^*E)[i-r],A(n+r)]$ induced by the left action on the Thom space $\Th(f^*E)$.

In particular, we have a right action of $\GL_r(k[X])$ on $\pi_{i}^{\A^1}(A)_n(X)$ given by the identification
\[
\pi_{i}^{\A^1}(A)_n(X)\cong \pi_{i}^{\A^1}(A)_n(X;\triv_X^r)
\]
combined with the left action of $\GL_r(k[X])$ on $\triv_X^r$. For $\alpha\in \pi_{i}^{\A^1}(A)_n(X),\, g\in \GL_r(k[X])$ we denote this action $\alpha\cdot g$. 
\end{definition}

\begin{lemma} \label{lm:GL_action}
Let $X$ be a smooth variety and $x\in X$ be a (not necessarily closed) point. Then for every $A\in\SH(k)$ we have
\begin{enumerate}
	\item
	$\alpha\cdot g=\alpha \cdot \det g$ for $\alpha\in \pi_{i}^{\A^1}(A)_n(\Spec \struct_{X,x}),\, g\in \GL_r(\struct_{X,x})$.
	\item
	$\alpha\cdot a= \alpha\cdot a^{-1}$ for $\alpha\in \pi_{i}^{\A^1}(A)_n(X),\, a\in k[X]^*$.
\end{enumerate}
\end{lemma}
\begin{proof}
For the first claim recall that over a local ring every matrix of determinant one is a product of elementary transvections. Then the claim follows from the well-known fact that the action given by such matrix is homotopy trivial (see, for example, \cite[Lemma~1]{An16}).

The second claim follows from \cite[Lemma~5]{An16}.
\end{proof}

\begin{definition} \label{def:twist}
Let $L$ be a line bundle over a smooth variety $X$. For $A\in\SH(k)$ denote by $\pi_{i}^{\A^1}(A)_n\otimes L$ the presheaf of abelian groups on $\Sm_k/X$ given by 
\[
(\pi_{i}^{\A^1}(A)_n\otimes L) (Y)= \pi_{i}^{\A^1}(A)_n(Y)\otimes_{\Z[k[Y]^*]} \Z[\Gamma(Y,f^*L)^0].
\]
Here $f\colon Y\to X$ is the structure map, $\Gamma(Y,f^*L)^0$ is the set of nowhere vanishing global sections of $f^*L$ (i.e. the set of trivializations of $f^*L$) and the action of $\GL_1(k[Y])=k[Y]^*$ on $\pi_{i}^{\A^1}(A)_n(Y)$ is given in Definition~\ref{def:Gm_action}. We denote by $\underline{\pi}_{i}^{\A^1}(A)_n\otimes L$ the associated Zariski sheaf on $\Sm_k/X$.
\end{definition}

\begin{remark} \label{rm:SLc_orientation}
It follows from Lemma~\ref{lm:GL_action} that $\pi_{i}^{\A^1}(A)_n\otimes L\cong \pi_{i}^{\A^1}(A)_n\otimes L^{-1}$, whence
\[
\pi_{i}^{\A^1}(A)_n\otimes (L'\otimes L^{\otimes 2})\cong \pi_{i}^{\A^1}(A)_n\otimes L',\quad
\underline{\pi}_{i}^{\A^1}(A)_n\otimes (L'\otimes L^{\otimes 2})\cong \underline{\pi}_{i}^{\A^1}(A)_n\otimes L'
\]
for line bundles $L,L'$ over a smooth variety $X$. 
\end{remark}

\begin{lemma}\label{lm:SL_oriented}
Let $E$ be a vector bundle over a smooth variety $X$. Then for every $A\in\SH(k)$ there exists a canonical isomorphism of sheaves 
\[
\underline{\pi}_{i}^{\A^1}(A)_n(-;E)\cong \underline{\pi}_{i}^{\A^1}(A)_n\otimes \det E.
\]
In particular, $\underline{\pi}_{i}^{\A^1}(A)_n(-;L)\cong \underline{\pi}_{i}^{\A^1}(A)_n\otimes L$ for a line bundle $L$ over $X\in\Sm_k$.
\end{lemma}
\begin{proof}
Put $r=\rank E$. For a regular morphism of smooth varieties $f\colon Y\to X$ consider the following homomorphisms of abelian groups.
\[
\xymatrix @C=0pt {
	& \ar[dl]_(0.6){\phi_Y} \pi_{i}^{\A^1}(A)_n(Y;\triv^r_X)\times \Z[\mathcal{I}so(f^*E,\triv_Y^r)] \ar[dr]^(0.6){\psi_Y} & \\
	\pi_{i}^{\A^1}(A)_n(Y;E) & & \pi_{i}^{\A^1}(A)_n(Y)\otimes_{\Z[k[Y]^*]} \Z[\Gamma(Y;\det f^*E)^0]
}
\]
Here 
\begin{itemize}
	\item
	$\mathcal{I}so(f^*E,\triv_Y^r)$ is the set of vector bundle isomorphisms $f^*E\cong \triv_Y^r$ (the set of trivializations of $f^*E$). 
	\item
	For $\alpha\in \pi_{i}^{\A^1}(A)_n(Y;\triv^r_X)$, $\theta\in \mathcal{I}so(f^*E,\triv_Y^r)$ we put $\phi_Y(\alpha,\theta)=\theta^*(\alpha)$ for the isomorphism 
	\[
	\theta^*\colon \pi_{i}^{\A^1}(A)_n(Y;\triv_X^r)\xrightarrow{\simeq} \pi_{i}^{\A^1}(A)_n(Y;E)
	\]
	induced by the corresponding isomorphism of Thom spaces $\Th(E) \xrightarrow{\simeq} \Th(\triv_Y^r) = \Th(f^*\triv_X^r)$.
	\item
	For $\alpha\in \pi_{i}^{\A^1}(A)_n(Y;\triv^r_X)$, $\theta\in \mathcal{I}so(f^*E,\triv_Y^r)$ we put $\psi_Y(a,\theta)=a\otimes(\det \theta)^{-1}$. Here we identify $\pi_{i}^{\A^1}(A)_n(Y;\triv^r_X)\cong \pi_{i}^{\A^1}(A)_n(Y)$.
\end{itemize}
These homomorphisms induce morphisms of presheaves of abelian groups on $\Sm_k/X$,
\[
\pi_{i}^{\A^1}(A)_n(-;E) \xleftarrow{\phi}\pi_{i}^{\A^1}(A)_n(-;\triv^r_X)\times \Z[\mathcal{I}so((-)^*E,\triv_{(-)}^r)] \xrightarrow{\psi} \pi_{i}^{\A^1}(A)_n\otimes \det E.
\]
Note that $\mathcal{I}so(f^*E,\triv_Y^r)$ has a canonical left action of $\GL_r(k[Y])$, while 
\[
\pi_{i}^{\A^1}(A)_n(Y;\triv_X^r)=\pi_{i}^{\A^1}(A)_n(Y;f^*\triv_X^r)=\pi_{i}^{\A^1}(A)_n(Y;\triv_Y^r)
\]
has a right action of $\GL_r(k[Y])$ given in Definition~\ref{def:Gm_action}. One easily checks that $\phi_Y(\alpha\cdot g,\theta)=\phi_Y(\alpha,g\theta)$ for $g\in \GL_r(k[Y])$. Moreover, Lemma~\ref{lm:GL_action} yields that for $g\in \GL_r(k[Y])$ the stalks of 
\[
\psi_Y(\alpha\cdot g, \theta)=(\alpha\cdot g)\otimes (\det \theta)^{-1}\]
and 
\[
\psi_Y(\alpha, g\theta)=\alpha\otimes (\det \theta)^{-1} (\det g)^{-1}=\alpha\cdot (\det g)^{-1}\otimes (\det \theta)^{-1}
\]
are equal. Thus we have the induced morphisms of sheaves of abelian groups on $\Sm_k/X$
\[
\underline{\pi}_{i}^{\A^1}(A)_n(-;E) \xleftarrow{\phi} \underline{\pi}_{i}^{\A^1}(A)_n(-;\triv^r_X)\otimes_{\Z[\GL_r(-)]} \Z[\mathcal{I}so((-)^*E,\triv_{(-)}^r)] \xrightarrow{\psi} \underline{\pi}_{i}^{\A^1}(A)_n\otimes \det E.
\]
One easily sees that these morphisms are stalk-wise isomorphisms, whence the claim.
\end{proof}

\begin{definition} \label{def:t-structure}
Recall \cite[Section 5.2]{Mor04} that the \textit{homotopy $t$-structure} on $\SH(k)$ is given by the following full subcategories:
\begin{gather*}
\SH(k)_{t\ge 0}=\{A\in \SH(k)\,|\, \underline{\pi}_{i}^{\A^1}(A)_n=0\, \text{for $i<0$, $n\in \Z$}\},\\
\SH(k)_{t\le 0}=\{A\in \SH(k)\,|\, \underline{\pi}_{i}^{\A^1}(A)_n=0\, \text{for $i>0$, $n\in \Z$}\}.
\end{gather*}
We denote the \textit{heart} of the homotopy $t$-structure as
\[
\heart=\SH(k)_{t\ge 0}\cap \SH(k)_{t\le 0}.
\]
For $l\in\Z$ we have an autoequivalence on $\heart$ given by
\[
M\mapsto M(l)=\Sigma_\Gm^l M.
\]

Put $\SH(k)_{t\ge n}=\SH(k)_{t\ge 0}[n],\, \SH(k)_{t\le n}=\SH(k)_{t\le 0}[n]$. Then for every $A\in \SH(k)$ there exists a canonical filtration
\[
\hdots \to A_{t\ge n} \to A_{t\ge n-1} \to \hdots \to A_{t\ge 1} \to A_{t\ge 0} \to A_{t\ge -1 }\to \hdots \to A
\]
with $A_{t\ge n}\in \SH(k)_{t\ge n}$. Denote 
\[
\Ht^\mathrm{t}_n(A)=\Cone (A_{t\ge n+1}\to A_{t\ge n})[-n]\in \heart.
\]
\end{definition}

\begin{notation}
Abusing the notation, for $M\in\heart$ we put 
\[
M_n(-;E)= \underline{\pi}^{\A^1}_0(M)_n (-;E),\quad M_n\otimes L= \underline{\pi}^{\A^1}_0(M)_n\otimes L \cong \underline{\pi}^{\A^1}_0(M)_n(-;L).
\]
\end{notation}

\section{Coniveau spectral sequence and Rost-Schmid complex}

In this section we recall the constructions of the coniveau spectral sequence and the Rost-Schmid complex (cf. \cite[Chapter~5]{Mor12}) and show that the Rost-Schmid complex allows one to compute sheaf cohomology of homotopy sheaves.

\begin{notation}
For a smooth variety $X$ and a smooth subvariety $Z\subset X$ denote
\[
\Lambda_Z^X=\det N_{Z/X}
\]
the determinant of the normal bundle of $Z$ in $X$.
\end{notation}

\begin{definition} \label{def:coniveau}
Let $X$ be a smooth variety of dimension $d$. Consider a sequence of open subsets
\[
\emptyset=U_\alpha^{(0)}\subset U_\alpha^{(1)}\subset \hdots \subset U_\alpha^{(d)}\subset U_\alpha^{(d+1)} = X
\]
satisfying 
\begin{enumerate}
	\item
	$Z_\alpha^{(p)}=U_\alpha^{(p+1)}-U_\alpha^{(p)}$ is smooth and equidimensional,
	\item
	$\dim Z_\alpha^{(p)} \le d - p$.
\end{enumerate}
Then for every $A\in \SH(k)$ there are long exact sequences 
\[
\hdots \to [U_\alpha^{(p+1)}/U_\alpha^{(p)}, A[m]] \to [U_\alpha^{(p+1)}, A[m]] \to  [U_\alpha^{(p)}, A[m]] \to [U_\alpha^{(p+1)}/U_\alpha^{(p)}, A[m+1]] \to \hdots
\]
Identify $U_\alpha^{(p+1)}/U_\alpha^{(p)}\cong \Th(N_{Z_\alpha^{(p)}/X})$ by the homotopy purity theorem \cite[Theorem 2.23]{MV99}. Then the long exact sequences could be rewritten as
\[
\hdots \to [\Th(N_{Z_\alpha^{(p)}/X}), A[m]] \to [U_\alpha^{(p+1)}, A[m]] \to [U_\alpha^{(p)}, A[m]] \to [\Th(N_{Z_\alpha^{(p)}/X}), A[m+1]] \to \hdots
\]
Rewriting $[\Th(N_{Z_\alpha^{(p)}/X}), A[m]]=\pi^{\A^1}_{p-m}(A)_{-p}(Z_\alpha^{(p)};N_{Z_\alpha^{(p)}/X})$, taking the colimit over all the considered sequences of open subsets and applying Lemma~\ref{lm:SL_oriented} we obtain
\begin{multline*}
\hdots \to \bigoplus_{x\in X^{(p)}} \underline{\pi}_{p-m}^{\A^1}(A)_{-p}(x; \Lambda_x^X) \to \colim_{\alpha} [U_\alpha^{(p+1)}, A[m]] \to \\
\to \colim_\alpha  [U_\alpha^{(p)}, A[m]] \to \bigoplus_{x\in X^{(p)}} \underline{\pi}_{p-m-1}^{\A^1}(A)_{-p}(x; \Lambda_x^X) \to \hdots
\end{multline*}
Here the sum is taken over the set of points of codimension $p$. These long exact sequences give rise to a \textit{coniveau spectral sequence}
\[
E_1^{p,q}= \bigoplus_{x\in X^{(p)}} \pi^{\A^1}_{-q}(A)_{-p}(x; \Lambda_x^X) \Rightarrow [X, A[p+q]].
\]
Let $U$ be an open subset of $X$ with closed complement $Z=X-U$ of dimension $d'$. If one considers only the subsets $U_\alpha^{(p)}$ containing $U$ and the corresponding sequences
\[
U/U=U_\alpha^{(d-d')}/U\subset U_\alpha^{(d-d'+1)}/U\subset \hdots \subset U_\alpha^{(d)}/U\subset U_\alpha^{(d+1)}/U = X/U
\]
then the above spectral sequence becomes
\[
E_1^{p,q}= \bigoplus_{\substack{x\in X^{(p)}\\ x\in Z}}\pi^{\A^1}_{-q}(A)_{-p}(x; \Lambda_x^X) \Rightarrow [X/U, A[p+q]].
\]
Note that for $U=\emptyset$ this spectral sequence coincides with the previous one. Over the spectrum of a field every vector bundle is trivial, thus we have non-canonical isomorphisms
\[
E_1^{p,q}= \bigoplus_{\substack{x\in X^{(p)}\\ x\in Z}}\pi^{\A^1}_{-q}(A)_{-p}(x; \Lambda_x^X)\simeq  \bigoplus_{\substack{x\in X^{(p)}\\ x\in Z}} \pi^{\A^1}_{-q}(A)_{-p}(x).
\]
\end{definition}

\begin{definition} \label{def:RS_complex_groups}
In the notation of Definition~\ref{def:coniveau} for $A=M\in \heart$ the first page of the coniveau spectral sequence is concentrated at the line $q=0$ and looks as follows:
\[
\bigoplus_{\substack{x\in X^{(0)}\\ x\in Z}} M_0(x)  \to \bigoplus_{\substack{x\in X^{(1)}\\ x\in Z}}M_{-1}(x; \Lambda_x^X)  \to \hdots \to \bigoplus_{\substack{x\in X^{(d-1)}\\ x\in Z}}M_{1-d}(x; \Lambda_x^X)  \to \bigoplus_{\substack{x\in X^{(d)}\\ x\in Z}}M_{-d}(x; \Lambda_x^X).
\]
We denote this complex $\RS^X(X,U;M)$ and refer to it as \textit{Rost-Schmid complex} (see \cite[Chapter~5]{Mor12} and especially \cite[Corollary~5.44]{Mor12}). For $U=\emptyset$ we put $\RS^X(X;M)=\RS^X(X,\emptyset;M)$.

Let $E$ be a rank $r$ vector bundle over $X\in\Sm_k$ and $z\colon X\to E$ be the zero section. Then the shifted Rost-Schmid complex $\RS^E(E,E-z(X);M(r))[r]$ looks as follows.
\begin{multline*}
\bigoplus_{x\in X^{(0)}} M_{0}(x; \det E|_x)  \to \bigoplus_{x\in X^{(1)}} M_{-1}(x; \Lambda_x^X\otimes \det E|_x)  \to \hdots \\
\hdots \to \bigoplus_{x\in X^{(d-1)}} M_{1-d}(x; \Lambda_x^X\otimes \det E|_x)  \to \bigoplus_{x\in X^{(d)}} M_{-d}(x; \Lambda_x^X\otimes \det E|_x).
\end{multline*}
We denote this complex $\RS^X(X;E;M)$. 

The above construction of the Rost-Schmid complex gives rise to the following complex of sheaves on the small Nisnevich site of $X$:
\begin{multline*}
\bigoplus_{x\in X^{(0)}}(i_x)_*M_{0}(-; \det E|_x)  \to \bigoplus_{x\in X^{(1)}} (i_x)_*M_{-1}(-; \Lambda_x^X\otimes \det E|_x)  \to \hdots \\
\hdots \to \bigoplus_{x\in X^{(d-1)}} (i_x)_*M_{1-d}(-; \Lambda_x^X\otimes \det E|_x)  \to \bigoplus_{x\in X^{(d)}} (i_x)_*M_{-d}(-; \Lambda_x^X\otimes \det E|_x).
\end{multline*}
Here $(i_x)_*M_{p}(-; \Lambda_x^X\otimes \det E|_x)$ is the Nisnevich direct image sheaf for the embedding $i_x\colon x\to X$. We denote this complex $\RS^X(-;E;M)$. Note that it follows from the construction that
\[
\RS^X(-;E;M) = \RS^X(-;\det E;M).
\]
\end{definition}

\begin{lemma}\label{lm:RS_cohomology}
Let $X$ be a smooth variety and $U\subset X$ be an open subset. Then for every $M\in \heart$ there are canonical isomorphisms
\[
[X/U, M[m]] \xrightarrow{\simeq} \coh^m(\RS^X(X,U;M)).
\]
In particular, for a rank $r$ vector bundle $E$ over $X\in\Sm_k$ we have
\[
[\Th(E), M(r)[r+m]] \xrightarrow{\simeq} \coh^m(\RS^X(X;\det E;M)).
\]
\end{lemma}
\begin{proof}
For $A=M$ the first page of the coniveau spectral sequence is concentrated at one line which by definition coincides with $\RS^X(X,U;M)$.
\end{proof}

\begin{lemma}\label{lm:direct_image_acyclic}
Let $X$ be a smooth variety and $x\in X$ be a (not necessarily closed) point. Denote $i\colon x \to X$ the corresponding embedding. Then for a sheaf $\mathcal{F}$ on the small Nisnevich site $x_{\Nis}$ and a closed subset $Z\subset X$ we have
\[
\coh^m_{\Zar,Z}(X;i_*\mathcal{F})=\coh^m_{\Nis,Z}(X;i_*\mathcal{F})=0
\]
for $m>0$. Here $i_*\mathcal{F}$ is the direct image sheaf in Nisnevich topology.
\end{lemma}
\begin{proof}
Follows from the Grothendieck spectral sequence for the composition of functors $\Gamma_Z(-) \circ i_*$ and exactness of $i_*$ and $\Gamma_Z(-)\circ i_*$.
\end{proof}

\begin{lemma}\label{lm:acyclic_resolution}
Let $L$ be a line bundle over $X\in\Sm_k$. Then for $M\in\heart$ the complex
\[
M_0(-;L) \to \RS^X(-;L;M)
\]
is a resolution in both Zariski and Nisnevich topologies on the corresponding small site of $X$. Here
\[
M_0(-;L) \to \mathcal{RS}^X_0(-;L;M)=\bigoplus_{x\in X^{(0)}} (i_x)_*M_{0}(-; L|_x)
\]
is induced by restriction homomorphisms
\[
M_0(U;L) \to \bigoplus_{x\in U^{(0)}} M_0(x;L|_x)
\]
for etale $U\to X$.
\end{lemma}
\begin{proof}
Lemma~\ref{lm:RS_cohomology} yields that for a local scheme $W=\Spec \struct_{X,x}$ (or $W=\Spec \struct_{X,x}^h$) we have
\[
\coh^m(\RS^X(W;L;M))\cong [\Th(L|_W), M(1)[m+1]]\cong \underline{\pi}^{\A^1}_m(M)_0(W).
\]
The last group is trivial for $m\neq 0$ since $M\in\heart$, whence the claim.
\end{proof}

\begin{theorem}[{cf.~\cite[Corollary~5.43]{Mor12}}] \label{thm:Hom_coh}
Let $E$ be a rank $r$ vector bundle over a smooth variety $X$ and let $U\subset X$ be an open subset. Put $Z=X-U$. Then for every $M\in \heart$ there exist canonical isomorphisms
\begin{enumerate}
	\item
	$[X/U,M[m]] \cong \coh^m(\RS^X(X,U;M))\cong \coh^m_{\Nis,Z}(X;M_0) \cong \coh^m_{\Zar,Z}(X;M_0)$,
	\item
	$[\Th(E),M(r)[r+m]] \cong \coh^m(\RS^X(X;\det E;M))\cong \coh^m_{\Nis}(X;M_0\otimes \det E) \cong \coh^m_{\Zar}(X;M_0\otimes \det E)$.
\end{enumerate}
In particular, 
\[
[X,M]\cong M_0(X)=\underline{\pi}^{\A^1}_0(M)_0(X),\quad [\Th(E),M(r)[r]]\cong M_0(X;E)=\underline{\pi}_0^{\A^1}(M)_0(X;E),
\]
i.e. $\pi_0^{\A^1}(M)_0$ and $\pi_0^{\A^1}(M)_0(-;E)$ are Nisnevich sheaves.
\end{theorem}
\begin{proof}
Follows from Lemmas~\ref{lm:RS_cohomology},~\ref{lm:direct_image_acyclic} and~\ref{lm:acyclic_resolution}.
\end{proof}

\begin{remark} \label{rm:coh_SLc}
It follows from the above theorem and Remark~\ref{rm:SLc_orientation} that if $M\in\heart$ is a commutative monoid, then the ring cohomology theory represented by $M$ is $\SL^c$-oriented in the sense of \cite[Definition~3.3]{PW10}. In particular, 
\[
\pi^{\A^1}_i(M)_n(X;E)\cong \pi^{\A^1}_i(M)_n(X;\det E).
\]
See also \cite[Theorem~4.2.7]{AH11} for the case of $M=\coh^t_0(\SSp)$.
\end{remark}

\begin{lemma} \label{lm:fibration_RS}
Let $X$ be a smooth variety of dimension $d$, $\rho\colon \widetilde{X}\to X$ be a Zariski locally trivial $\A^s$-fibration, $E$ be a rank $r$ vector bundle over $\widetilde{X}$, $L$ be a line bundle over $X$ and $\theta\colon\det E\cong \rho^*L$ be an isomorphism. Then for every $M\in\heart$ there exists an isomorphism
\[
\Theta\colon \coh^d(\RS^X(X;L;M)) \xrightarrow{\simeq} [\Th(E),M(r)[r+d]]
\]
such that for every rational point $y\in X(k)$ the following diagram commutes.
\[
\xymatrix{
	\pi_0^{\A^1}(M)_{-d}(\widetilde{y};\Lambda_{\widetilde{y}}^E) \ar[r]^(0.45){\phi}_(0.45)\simeq \ar[d]_{\Psi}^\cong &  M_{-d}(y; \Lambda_y^X\otimes L_y) \ar[r]^{i_y}  & \coh^d(\RS^X(X;L;M)) \ar[d]^\cong_\Theta \\
	\pi_0^{\A^1}(M)_{-d}(\widetilde{y};N_{\widetilde{y}/E}) \ar[r]^(0.4){d}_(0.4)\simeq  &  [E/(E-\widetilde{y}),M(r)[r+d]] \ar[r]^(0.52)Q &  [\Th(E),M(r)[r+d]]
}
\]
Here
\begin{itemize}
	\item
	$\widetilde{y}=z\circ \rho^{-1}(y)\subset E$ for the zero section $z\colon \widetilde{X}\to E$,
	\item
	$i_y$ is induced by the inclusion $M_{-d}(y; \Lambda_y^X\otimes L_y)\to \bigoplus\limits_{x\in X^{(d)}} M_{-d}(x; \Lambda_x^X\otimes L_x)$ to the summand corresponding to the point $y$,
	\item
	$Q$ is induced by the quotient map $\Th(E)\to E/(E-\widetilde{y})$,
	\item
	$d$ is induced by the homotopy purity theorem,
	\item
	$\Psi$ is given by Lemma~\ref{lm:SL_oriented},
	\item
	$\phi$ is induced by the projection $\rho\colon \widetilde{y} \to y$ and the isomorphisms $\Lambda^E_{\widetilde{y}}\cong \Lambda^{\widetilde{X}}_{\widetilde{y}}\otimes \det E|_{\widetilde{y}} \cong \rho^*(\Lambda_y^X\otimes L_y)$.
\end{itemize}
\end{lemma}
\begin{proof}
Consider a sequence of open subsets
\[
\emptyset=U_\alpha^{(0)}\subset U_\alpha^{(1)}\subset \hdots \subset U_\alpha^{(d)}\subset U_\alpha^{(d+1)} = X
\]
as in Definition~\ref{def:coniveau}. Put $\widetilde{U}_\alpha^{(p)}=\rho^{-1} (U_\alpha^{(p)})$ and $\widetilde{W}_\alpha^{(p)}=\pi^{-1}(\widetilde{U}_\alpha^{(p)})\cup (E-z(\widetilde{X}))$ for the canonical projection $\pi\colon E\to \widetilde{X}$ and the zero section $z\colon \widetilde{X}\to E$. We have a sequence of open subsets
\[
E-z(\widetilde{X})=\widetilde{W}_\alpha^{(0)}\subset \widetilde{W}_\alpha^{(1)}\subset \hdots \subset \widetilde{W}_\alpha^{(d)}\subset \widetilde{W}_\alpha^{(d+1)} = E.
\]
For $M\in\heart$ we obtain a family of long exact sequences
\begin{multline*}
\hdots \to [\Th(N_{\widetilde{Z}_\alpha^{(p)}/E}), M(r)[m]] \to [\widetilde{W}_\alpha^{(p+1)}/\widetilde{W}_\alpha^{(0)}, M(r)[m]] \to \\
\to  [\widetilde{W}_\alpha^{(p)}/\widetilde{W}_\alpha^{(0)}, M(r)[m]] \to [\Th(N_{\widetilde{Z}_\alpha^{(p)}/E}), M(r)[m+1]] \to \hdots
\end{multline*}
where $\widetilde{Z}_\alpha^{(p)}=z\circ\rho^{-1}(Z_\alpha^{(p)})$ with $Z_\alpha^{(p)}=U_\alpha^{(p+1)}-U_\alpha^{(p)}$. 

Take the colimit over the families of $U_\alpha^{(p)}$. The resulting long exact sequences can be organized in the following version of the coniveau spectral sequence (cf. Definition~\ref{def:coniveau}):
\[
E_1^{p,q}= \bigoplus_{x\in X^{(p)}}
\pi^{\A^1}_{r-q}(M)_{-p}(\widetilde{x};N_{\widetilde{x}/E}) \Rightarrow [\Th(E), M(r)[p+q]],
\]
where $\widetilde{x}=z\circ\rho^{-1}(x)\cong \A^{s}_{k(x)}$. Since every vector bundle over an affine space is trivial and $M\in\heart$ we have
\[
\pi^{\mathbb{A}^1}_{r-q}(M)_{-p}(\widetilde{x};N_{\widetilde{x}/E})\cong \pi^{\mathbb{A}^1}_{r-q}(M)_{-p}(\widetilde{x})\cong \pi^{\mathbb{A}^1}_{r-q}(M)_{-p}(x)=0
\]
for $q\neq r$. Thus the first page of the spectral sequence is concentrated at the line $q=r$ and takes the form:
\[
\bigoplus_{x\in X^{(0)}} M_{0}(\widetilde{x};N_{\widetilde{x}/E})  \to \bigoplus_{x\in X^{(1)}} M_{-1}(\widetilde{x};N_{\widetilde{x}/E})  \to
\hdots \to  \bigoplus_{x\in X^{(d-1)}} M_{1-d}(\widetilde{x};N_{\widetilde{x}/E})  \to \bigoplus_{x\in X^{(d)}} M_{-d}(\widetilde{x};N_{\widetilde{x}/E}).
\]

We have $\det N_{\widetilde{x}/E} \cong \rho^* \Lambda_{x}^X \otimes \det E|_{\widetilde{x}} \cong \rho^*(\Lambda_{x}^X \otimes L|_x)$, thus we may rewrite this complex as
\begin{multline*}
\bigoplus_{x\in X^{(0)}} M_{0}(\widetilde{x};\rho^*L|_x)  \to \bigoplus_{x\in X^{(1)}} M_{-1}(\widetilde{x};\rho^*(\Lambda_{x}^{X} \otimes L|_x))  \to \hdots \\
\hdots \to  \bigoplus_{x\in X^{(d-1)}} M_{1-d}(\widetilde{x};\rho^*(\Lambda_{x}^{X} \otimes L|_x))  \to \bigoplus_{x\in X^{(d)}} M_{-d}(\widetilde{x};\rho^*(\Lambda_{x}^{X} \otimes L|_x)).
\end{multline*}
The projection $\rho \colon \widetilde{x}\to x$ is a homotopy equivalence, thus we may further rewrite it as
\begin{multline*}
\bigoplus_{x\in X^{(0)}} M_{0}(x;L|_x)  \to \bigoplus_{x\in X^{(1)}} M_{-1}(x;\Lambda_{x}^{X} \otimes L|_x)  \to \hdots \\
\hdots \to  \bigoplus_{x\in X^{(d-1)}} M_{1-d}(x;\Lambda_{x}^{X} \otimes L|_x)  \to \bigoplus_{x\in X^{(d)}} M_{-d}(x;\Lambda_{x}^{X} \otimes L|_x).
\end{multline*}
One can easily check that this complex coincides with $\RS^X(X;L;M)$. The claim follows.
\end{proof}

\section{Rost-Schmid complex for $C/D$}

In this section we present a variant of the Rost-Schmid complex for $C/D$ with $C$ being a smooth curve and $D\subset C^{(1)}$ being a finite collection of closed points. This complex, roughly speaking, is the cone for an appropriate morphism from the Rost-Schmid complex of $C$ to the Rost-Schmid complex of $D$.

\begin{definition}\label{def:RS_CD}
Let $\pi\colon E\to C$ be a rank $r$ vector bundle over a smooth curve $C$ with the zero section $z\colon C\to E$. Let $D\subset C^{(1)}$ be a finite collection of closed points. For an open subset $D\subset U_\alpha\subset C$ put $U^E_\alpha=\pi^{-1}(U_\alpha)\cup (E-z(C))$. Then for every $M\in\heart$ there is a long exact sequence
\begin{multline*}
\hdots \to [E/U^E_\alpha,M(r)[r+m]] \to [\Th(E),M(r)[r+m]] \to \\
\to  [U^E_\alpha/(E-z(C)), M(r)[r+m]]\to [E/U^E_\alpha,M(r)[r+m+1]] \to \hdots
\end{multline*}
Taking the colimit over $U_\alpha$ and applying Lemma~\ref{lm:SL_oriented} we obtain a long exact sequence
\begin{multline*}
\hdots \to \bigoplus_{\substack{x\in C^{(1)} \\ x\not\in D}} \pi^{\A^1}_{-m+1}(M)_{-1}(x; \Lambda_x^C\otimes \det E|_x)
\to [\Th(E),M(r)[r+m]] \to \\
\to \pi^{\A^1}_{-m}(M)_{0}(\Spec \struct_{C,D};\det E) \xrightarrow{\partial} \bigoplus_{\substack{x\in C^{(1)} \\ x\not\in D}} \pi^{\A^1}_{-m}(M)_{-1}(x; \Lambda_x^C\otimes \det E|_x) \to \hdots
\end{multline*}
The corresponding version of the Rost-Schmid complex consists of two terms:
\[
M_0(\Spec \struct_{C,D};\det E)  \xrightarrow{\partial} \bigoplus_{\substack{x\in C^{(1)} \\ x\not\in D}} M_{-1}(x; \Lambda_x^C\otimes \det E|_x).
\]
We consider the following modified version of the Rost-Schmid complex,
\[
M_0(\Spec \struct_{C,D};\det E)  \xrightarrow{(\partial, i_D^*)} \bigoplus_{\substack{x\in C^{(1)} \\ x\not\in D}} M_{-1}(x; \Lambda_x^C\otimes \det E|_x)\oplus \bigoplus_{x\in D} M_{0}(x; \det E|_x),
\]
where $i_D^*$ is the restriction morphism for the embedding $i_D\colon D\to \Spec \struct_{C,D}$. This complex is denoted $\RS^C(C,D;E;M)$. Note that 
\[
\RS^C(C,D;E;M)=\RS^C(C,D;\det E;M).
\]
\end{definition}

\begin{lemma} \label{lm:RSCD_coh}
Let $\pi\colon E\to C$ be a rank $r$ vector bundle over a smooth curve $C$ with the zero section $z\colon C\to E$. Let $D\subset C^{(1)}$ be a finite collection of closed points. Then for $M\in\heart$ and $m\in \Z$ there are canonical isomorphisms
\[
[\Th(E)/\Th(E|_D),M(r)[r+m]]\cong \coh^m(\RS^C(C,D;\det E;M)).
\]
\end{lemma}
\begin{proof}
It follows from the discussion in Definition~\ref{def:RS_CD} that there is a long exact sequence
\begin{multline*}
\hdots \to \bigoplus_{\substack{x\in C^{(1)} \\ x\not\in D}} \pi^{\A^1}_{-m+1}(M)_{-1}(x; \Lambda_x^C\otimes \det E|_x)
\to [\Th(E),M(r)[r+m]] \to \\
\to \pi^{\A^1}_{-m}(M)_{0}(\Spec \struct_{C,D};\det E) \xrightarrow{\partial} \bigoplus_{\substack{x\in C^{(1)} \\ x\not\in D}} \pi^{\A^1}_{-m}(M)_{-1}(x; \Lambda_x^C\otimes \det E|_x) \to \hdots
\end{multline*}
Note that $\pi^{\A^1}_{-m}(M)_{-1}(x; \Lambda_x^C\otimes \det E|_x)=0$ for $m\neq 0$ since $M\in\heart$. Moreover, 
\[
\pi^{\A^1}_{-m}(M)_{0}(\Spec \struct_{C,D};\det E)\cong  \coh^{m}( \RS^C (\Spec \struct_{C,D};E;M)) \cong \coh^{m}_{\Nis}(\Spec \struct_{C,D};M_0\otimes \det E)
\]
by Theorem~\ref{thm:Hom_coh}. The cohomology of the Rost-Schmid complex vanishes for $m\neq 0,1$ since the complex consists of two terms. The Nisnevich cohomology group vanishes for $m=1$ by Lemma~\ref{lm:coh_semiloc}. Thus $\pi^{\A^1}_{-m}(M)_{0}(\Spec \struct_{C,D};\det E)=0$ for $m\neq 0$. Thus we have an exact sequence
\begin{multline*} 
0 \to [\Th(E),M(r)[r]] \to M_{0}(\Spec \struct_{C,D};\det E)  \xrightarrow{\partial} \\
\xrightarrow{\partial}  \bigoplus_{\substack{x\in C^{(1)} \\ x\not\in D}} M_{-1}(x; \Lambda_x^C\otimes \det E|_x) \to [\Th(E),M(r)[r+1]] \to 0 
\end{multline*}
and $[\Th(E),M(r)[r+m]]=0$ for $m\neq 0,1$.

The natural embedding $\Th(E|_D)\to \Th(E)$ gives rise to a long exact sequence
\begin{multline*}
\hdots \to [\Th(E)/\Th(E|_D),M(r)[r+m]] \to [\Th(E),M(r)[r+m]] \to \\
\to [\Th(E|_D),M(r)[r+m]] \to [\Th(E)/\Th(E|_D),M(r)[r+m+1]]  \to\hdots
\end{multline*}
By the above and since $[\Th(E|_D),M(r)[r+m]]=\pi^{\A^1}_{-m}(M)_0(D;E|_D)$ vanishes for $m\neq 0$ we have 
\[
[\Th(E)/\Th(E|_D),M(r)[r+m]]=0
\]
for $m\neq 0,1$.

For $m=0$ consider the following diagram.
\[
\xymatrix @C=0.8pc{
	& 0 \ar[d] & & \\
	& [\Th(E)/\Th(E|_D),M(r)[r]] \ar[d] & & \\
	0 \ar[r] & [\Th(E),M(r)[r]] \ar[r] \ar[d] & M_{0}(\Spec \struct_{C,D};\det E) \ar[dl]^{i_D^*} \ar[r]^(0.45){\partial} & \bigoplus\limits_{\substack{x\in C^{(1)} \\ x\not\in D}} M_{-1}(x; \Lambda_x^C\otimes \det E|_x) \\
	& M_0(D;\det E|_D) & &
}
\]
Here the vertical exact sequence is the one associated to the embedding $\Th(E|_D)\to \Th(E)$ while the horizontal exact sequence is a part of the coniveau one. The triangle commutes since all the involved morphisms are given by restriction and Lemma~\ref{lm:SL_oriented}. The isomorphism
\[
[\Th(E)/\Th(E|_D),M(r)[r]]\cong \coh^0(\RS^C(C,D;\det E;M))
\]
follows via diagram chase.

For $m=1$ consider the following commutative diagram.
\[
\xymatrix @R=1.0pc{
	[\Th(E),M(r)[r]] \ar[r]^= \ar[d] & [\Th(E),M(r)[r]] \ar[d] \\
	M_{0}(\Spec \struct_{C,D};\det E) \ar[d]_\partial \ar[r]^{i_D^*} & M_0(D;\det E|_D) \ar[d]\\
	\bigoplus\limits_{\substack{x\in C^{(1)} \\ x\not\in D}} M_{-1}(x; \Lambda_x^C\otimes \det E|_x)
	\ar@{}[d]^(.15){}="a"^(.85){}="b" \ar "a";"b"
	\ar[r] &  [\Th(E)/\Th(E|_D),M(r)[r+1]] \ar[d]  \\
	[\Th(E),M(r)[r+1]] \ar[r]^= \ar[d] & [\Th(E),M(r)[r+1]] \ar[d] \\
	0 & 0
}
\]
Here the exact sequence in the first row is a part of the coniveau sequence while the exact sequence in the second row is the one associated to the embedding $\Th(E|_D)\to \Th(E)$. The horizontal morphisms are induced by the following morphism of triangles (to be more precise one needs to take a limit over $U_\alpha$ as in Definition~\ref{def:RS_CD})
\[
\xymatrix{
	\Th(E|_D) \ar[r] \ar[d]^{i_D} & \Th(E) \ar[r] \ar[d]^{=} & \Th(E)/\Th(E|_D) \ar[r] \ar[d] & \Th(E|_D)[1] \ar[d] \\
	\Th(E|_{\Spec \struct_{C,D}}) \ar[r] & \Th(E) \ar[r] & \Th(E)/\Th(E|_{\Spec \struct_{C,D}}) \ar[r] & \Th(E|_{\Spec \struct_{C,D}})[1]
}
\]
The claim follows by diagram chase.
\end{proof}

\begin{definition} \label{def:MCD}
Let $L$ be a line bundle over a smooth curve $C$ and $D\subset C^{(1)}$ be a finite collection of closed points. For $M\in\heart$ consider the following sheaves on the small Nisnevich site of $C$. 
\begin{itemize}
	\item
	Let $(M_0)_{C}^D(-; L)$ be the sheaf given by
	\[
	(M_0)_{C}^D(U; L)= M_0(\Spec \struct_{U,D_U}; L)
	\]
	for an etale morphism $U\to C$ and $D_U=U\times_C D$. It follows from Lemma~\ref{lm:RS_cohomology} that this sheaf coincides with the kernel
	\[
	(M_0)_{C}^D(-;L)=\ker \left[\bigoplus_{x\in C^{(0)}} (i_x)_* M_0(-;L|_x) \xrightarrow{\partial} \bigoplus_{x\in D} (i_x)_* M_{-1}(-;\Lambda_x^C\otimes L|_x) \right].
	\]
	Here on the right-hand side we have the Rost-Schmid complex of sheaves 
	\[
	\RS^{\Spec \struct_{C,D}}(-;L|_{\Spec \struct_{C,D}};M)
	\]
	introduced in Definition~\ref{def:RS_complex_groups} pushed forward to $C$.
	\item
	Let $(M_0)_{C,D}(-; L)$ be the sheaf given by
	\[
	(M_0)_{C,D}(U; L)= \ker \left[ M_0(U; L)\xrightarrow{i_{D_U}^*} M_0(D_U; L) \right]
	\]
	for an etale morphism $U\to C$, $D_U=U\times_C D$ and the natural embedding $i_{D_U}\colon D_U\to U$. In other words, $(M_0)_{C,D}(-; L)= \ker \left[ M_0(-; L)\xrightarrow{i_{D}^*} (i_D)_*M_0(-; L|_D) \right]$.
\end{itemize}
Note that by Lemma~\ref{lm:SL_oriented} we have a canonical isomorphism
\[
(M_0)_{C,D}(-; L) \cong  (M_0\otimes L )_{C,D} = \ker \left[ M_0\otimes L\xrightarrow{i_{D}^*} (i_D)_*(M_0\otimes L|_D) \right].
\]

The discussion in Definition~\ref{def:RS_CD} gives rise to the following two term complex of sheaves on the small Nisnevich site of $C$. 
\[
(M_0)_C^D(-;L) \xrightarrow{(\partial,i_D^*)} \bigoplus_{\substack{x\in C^{(1)} \\ x\not\in D}} (i_x)_* M_{-1}(-; \Lambda_x^C\otimes L|_x)\oplus \bigoplus_{x\in D} (i_x)_*M_{0}(-; L|_x)
\]
Here $(i_x)_*$ is the direct image in Nisnevich topology. This complex is denoted $\RS^C(-,D;E;M)$.
\end{definition}

\begin{lemma}\label{lm:MCD_acyclic}
Let $L$ be a line bundle over a smooth curve $C$ and $D\subset C^{(1)}$ be a finite collection of closed points. Then for $M\in\heart$ and $m>0$ we have
\[
\coh^m_\Nis(C,(M_0)_{C}^D(-; L))=\coh^m_\Zar(C,(M_0)_{C}^D(-; L))=0.
\]
\end{lemma}
\begin{proof}
Let $U\to C$ be an etale morphism and put $D_U=U\times _C D$. We claim that the sequence
\[
0\to (M_0)_{C}^D(U;L)\to \bigoplus_{x\in U^{(0)}} M_0(x;L|_x) \xrightarrow{\partial} \bigoplus_{x\in D_U} M_{-1}(x;\Lambda_x^U\otimes L|_x) \to 0
\]
that arises from Definition~\ref{def:MCD} is exact. Indeed, by the definition $\ker \partial = (M_0)_{C}^D(U;L)$ while by Lemma~\ref{lm:RS_cohomology} we have $\coker \partial \cong \coh^1_{\Nis}(\Spec \struct_{U,D_U}; M_0\otimes L)$ and this group vanishes by Lemma~\ref{lm:coh_semiloc}.

Hence
\[
0\to (M_0)_{C}^D(-;L)\to \bigoplus_{x\in C^{(0)}} (i_x)_* M_0(-;L|_x) \xrightarrow{\partial} \bigoplus_{x\in D} (i_x)_* M_{-1}(-;\Lambda_x^C\otimes L|_x) \to 0
\]
is an exact sequence of sheaves. The claim follows from the associated long exact sequence of cohomology groups, Lemma~\ref{lm:direct_image_acyclic} and the aforementioned surjectivity on global sections.
\end{proof}

\begin{lemma} \label{lm:RSCD_res}
Let $L$ be a line bundle over a smooth curve $C$ and $D\subset C^{(1)}$ be a finite collection of closed points. Then for $M\in\heart$ the complex
\[
(M_0)_{C,D}(-;L)\to \RS^C(-,D;L;M)
\]
is a resolution in Nisnevich topology on the small site of $C$. Here the morphism
\[
(M_0)_{C,D}(-;L) \to \mathcal{RS}_0^C(-,D;L;M) = (M_0)_{C}^D(-; L)
\]
is induced by restriction homomorphisms
\[
M_0(U;L)\to M_0(\Spec \struct_{U,D_U};L) =(M_0)_{C}^D(U; L),
\]
where $U\to C$ is etale and $D_U=U\times_C D$.

Moreover, if for every closed point $x\in C$ the restriction homomorphism $M_0(\Spec \struct_{C,x})\to M_0(x)$ is surjective, then the complex is a resolution in Zariski topology.
\end{lemma}
\begin{proof}
For a local scheme $W=\Spec \struct_{C,x}$ (or $W=\Spec \struct_{C,x}^h$) and $D_W=W\times_C D$ Lemma~\ref{lm:RSCD_coh} yields an isomorphism
\[
\coh^m(\RS^C(W,D;L;M))\cong [\Th(L|_W)/\Th(L|_{D_W}), M(1)[m+1]].
\]
The embedding $\Th(L|_{D_W})\to \Th(L|_{W})$ gives rise to an exact sequence
\begin{multline*}
0\to [\Th(L|_W)/\Th(L|_{D_W}), M(1)[1]]\to [\Th(L|_W), M(1)[1]] \to \\
\to [\Th(L|_{D_W}), M(1)[1]] \to [\Th(L|_W)/\Th(L|_{D_W}), M(1)[2]] \to 0
\end{multline*}
Here all the other terms of the long exact sequence vanish since $M\in \heart$. Thus 
\[
\coh^0(\RS^C(W,D;L;M))\cong (M_0)_{C,D}(W)
\]
and one can easily check that this isomorphism is compatible with the restriction homomorphism from the statement of the lemma.

It remains to show that $[\Th(L|_W), M(1)[1]] \to [\Th(L|_{D_W}), M(1)[1]]$ is surjective. Trivializing $L|_W$ we see that it is sufficient to show that the restriction homomorphism
\[
M_0(W) \to M_0(D_W)
\]
is surjective. The case when $D_W=\emptyset$ is trivial. Thus we may assume that $D_W=x$ is the closed point of $W=\Spec \struct_{C,x}$ (or $W=\Spec \struct_{C,x}^h$). The claim for the Zariski topology ($W=\Spec \struct_{C,x}$) follows from the assumption. For the Nisnevich topology one notices that $\Spec \struct_{C,x}^h\cong \Spec (\A^1_{k(x)})_{(0)}^h$, thus the embedding $x \to \Spec \struct_{C,x}^h$ has a section.
\end{proof}

\begin{theorem} \label{thm:Hom_coh_CD}
Let $E$ be a rank $r$ vector bundle over a smooth curve $C$ and let $D\subset C^{(1)}$ be a finite collection of closed points. Then for $M\in\heart$ there exist natural isomorphisms
\[
[\Th(E)/\Th(E|_D),M(r)[r+m]] \cong \coh^m(\RS^C(C,D;\det E;M))\cong \coh^m_{\Nis}(X;(M_0 \otimes \det E)_{C,D}).
\]

Moreover, if for every closed point $x\in C$ the restriction homomorphism $M_0(\Spec \struct_{C,x})\to M_0(x)$ is surjective, then
\[
\coh^m_{\Nis}(X;(M_0\otimes \det E)_{C,D}) \cong \coh^m_{\Zar}(X;(M_0\otimes \det E)_{C,D}).
\]
\end{theorem}
\begin{proof}
The first isomorphism follows from Lemma~\ref{lm:RSCD_coh}. The other isomorphisms follow from the fact that $\RS^C(-,D;\det E;M)$ provides an acyclic resolution for $(M_0\otimes \det E)_{C,D}$ by Lemmas~\ref{lm:direct_image_acyclic},~\ref{lm:MCD_acyclic} and~\ref{lm:RSCD_res}.
\end{proof}

\begin{lemma} \label{lm:fibration_RS_CD}
Let $C$ be a smooth curve, $D\subset C^{(1)}$ be a finite collection of closed points, $\rho\colon \widetilde{C}\to C$ be a Zariski locally trivial $\A^s$-fibration, $E$ be a rank $r$ vector bundle over $\widetilde{C}$, $L$ be a line bundle over $C$ and $\theta\colon\det E\cong \rho^*L$ be an isomorphism. Put $\widetilde{D}=\rho^{-1}(D)$. Then for every $M\in\heart$ there exists an isomorphism
\[
\Theta\colon \coh^1(\RS^C(C,D;L;M)) \xrightarrow{\simeq} [\Th(E)/\Th(E|_{\widetilde{D}}),M(r)[r+1]]
\]
such that for every rational point $y\in C(k), y\not\in D,$ the following diagram commutes.
\[
\xymatrix{
	\pi_0^{\A^1}(M)_{-1}(\widetilde{y};\Lambda_{\widetilde{y}}^E) \ar[r]^(0.45){\phi}_(0.45)\simeq \ar[d]_{\Psi}^\cong &  M_{-1}(y; \Lambda_y^C\otimes L_y) \ar[r]^{i_y}  & \coh^1(\RS^C(C,D;L;M)) \ar[d]^\cong_\Theta \\
	\pi_0^{\A^1}(M)_{-1}(\widetilde{y};N_{\widetilde{y}/E}) \ar[r]^(0.4){d}_(0.4)\simeq  &  [E/(E-\widetilde{y}),M(r)[r+1]] \ar[r]^(0.45)Q &  [\Th(E)/\Th(E|_{\widetilde{D}}),M(r)[r+1]]
}
\]
Here
\begin{itemize}
	\item
	$\widetilde{y}=z\circ \rho^{-1}(y)\subset E$ for the zero section $z\colon \widetilde{C}\to E$,
	\item
	$i_y$ is induced by the inclusion $M_{-1}(y; \Lambda_y^C\otimes L_y)\to \bigoplus\limits_{\substack{x\in C^{(1)},\\ x\not\in D}} M_{-1}(x; \Lambda_x^C\otimes L_x)$ to the summand corresponding to the point $y$,
	\item
	$Q$ is induced by the quotient $\Th(E)/\Th(E|_{\widetilde{D}})\to E/(E-\widetilde{y})$,
	\item
	$d$ is induced by the homotopy purity theorem,
	\item
	$\Psi$ is given by Lemma~\ref{lm:SL_oriented},
	\item
	$\phi$ is induced by projection $\rho\colon \widetilde{y} \to y$ and isomorphisms $\Lambda^E_{\widetilde{y}}\cong \Lambda^{\widetilde{C}}_{\widetilde{y}}\otimes \det E|_{\widetilde{y}} \cong \rho^*(\Lambda_y^C\otimes L_y)$.
\end{itemize}
\end{lemma}
\begin{proof}
The proof is similar to the one given for Lemma~\ref{lm:fibration_RS}.
\end{proof}

\section{The dual spectrum and the zeroth stable homotopy group}

In this section we apply the Atiyah duality, basic properties of homotopy $t$-structure and developed techniques in order to obtain a cohomological description of the zeroth stable motivic homotopy groups.

\begin{definition}
For $A\in \SH(k)$ denote $A^\vee=\iHom_{\SH(k)}(A,\SSp)$ the \textit{dual spectrum}.
\end{definition}

\begin{definition} \label{def:Thom_stable_normal}
Let $X$ be a smooth variety. By Jouanolou device (see \cite{J73}, \cite[\S~4]{We89}) there is a morphism $\rho\colon \widetilde{X}\to X$ such that $\widetilde{X}$ is affine and $\rho$ is a locally trivial $\A^s$-bundle in Zariski topology. For the tangent bundle $T_X$ choose an isomorphism $\rho^*T_X\oplus E\cong \triv^N_{\widetilde{X}}$. Put
\[
\Th(-T_X)=\Sigma_T^{-N}\Sigma^\infty_\T \Th(E).
\]
One can show that up to a canonical isomorphism $\Th(-T_X)$ does not depend on the choices made (see also \cite[Section~3]{H05}).
\end{definition}

\begin{theorem}[{\cite[Theorem~A.1]{H05}, \cite[{Corollary~6.13}]{H15}}] \label{thm:Atiyah_duality}
Let $X$ be a smooth projective variety. Then $\Sigma_\T^\infty X_+$ is strongly dualizable and there exists a canonical isomorphism
\[
(\Sigma_\T^\infty X_+)^\vee \cong \Th(-T_X)
\]
that gives rise to the evaluation isomorphism
\[
[\SSp, \Sigma_\T^\infty X_+ (l)]\cong [\Th(-T_X),\SSp(l)].
\]
\end{theorem}

\begin{remark} \label{rem:U_dualizable}
Let $\mathcal{V}$ be a vector bundle over a smooth projective variety $X$. Then the following generalization of Theorem~\ref{thm:Atiyah_duality} holds (see {\cite[Remark~1]{H05}} or the proof of \cite[{Corollary~6.13}]{H15}): $\Sigma_\T^\infty \Th(\mathcal{V})$ is strongly dualizable and 
\[
(\Sigma_\T^\infty \Th(\mathcal{V}))^\vee\cong \Th(-(T_X\oplus \mathcal{V}))
\]
with $\Th(-(T_X\oplus \mathcal{V}))$ defined along the lines of Definition~\ref{def:Thom_stable_normal}. Thus for a smooth closed subvariety $Z$ of a smooth projective variety $X$ the triangle
\[
\Sigma_\T^\infty (X-Z)_+\to \Sigma_\T^\infty X_+ \to \Sigma_\T^\infty \Th(N_{Z/X}) \to \Sigma_\T^\infty (X-Z)_+[1]
\]
together with the fact that the subcategory of strongly dualizable objects is triangulated yields that $X-Z$ is strongly dualizable.
\end{remark}

\begin{lemma} \label{lm:dual_point}
In the notation of Definition~\ref{def:Thom_stable_normal} and Theorem~\ref{thm:Atiyah_duality} let $i_x\colon \Spec k\to X$ be a rational point. Then under the isomorphism 
\[
[\SSp, \Sigma_\T^\infty X_+]\cong [\Th(-T_X),\SSp]
\]
morphism
\[
\Sigma^\infty_T (i_x)_+\colon \SSp \to \Sigma_T^\infty X_+
\]
corresponds to the composition
\begin{multline*}
\Th(-T_X)=\Sigma_T^{-N}\Sigma^\infty_\T \Th(E)\to \Sigma_T^{-N}\Sigma^\infty_\T E/(E-z(\widetilde{x}))\xrightarrow{\simeq} \\
\xrightarrow{\simeq} \Sigma_T^{-N}\Sigma^\infty_\T \Th(E|_{z(\widetilde{x})}\oplus N_{\widetilde{x}/\widetilde{X}})\xrightarrow{\simeq}  \Sigma_T^{-N}\Sigma^\infty_\T \Th(\triv^N_{z(\widetilde{x})}) \xrightarrow{\simeq} \SSp.
\end{multline*}
Here
\begin{itemize}
	\item
	$z\colon \widetilde{X}\to E$ is the zero section;
	\item
	$\widetilde{x}=\rho^{-1}(x)$ for the rational point $x=i_x(\Spec k)$ and projection $\rho\colon \widetilde{X}\to X$;
	\item
	the first morphism is given by the quotient map;
	\item
	the second morphism is given by the homotopy purity theorem;
	\item
	the third morphism is given by the isomorphisms $N_{\widetilde{x}/\widetilde{X}}\cong (\rho^*T_X)|_{\widetilde{x}}$ and $E\oplus \rho^*T_X\cong \triv^N_{\widetilde{X}}$,
	\item
	the last isomorphism is given by the projection $\A^s\cong \widetilde{x}\to \Spec k$.
\end{itemize}
\end{lemma}
\begin{proof}
Follows from the proofs of {\cite[Theorem~A.1]{H05}} or \cite[{Corollary~6.13}]{H15}.
\end{proof}

\begin{lemma} \label{lm:dual_coh_dim}
Let $\mathcal{X}\in\SH(k)_{t\ge 0}$ be a strongly dualizable object. Then $\mathcal{X}^\vee \in {}^\perp\SH(k)_{t\ge 1}$.
\end{lemma}
\begin{proof}
For $A\in \SH(k)_{t\ge 1}$ we have
\[
[\mathcal{X}^\vee,A]\cong [\SSp,\mathcal{X}\wedge A].
\]
The last group vanishes since $\mathcal{X}\wedge A\in \SH(k)_{t\ge 1}$ by \cite[Lemma~3.2]{ALP15}.
\end{proof}

\begin{lemma} \label{lm:cohomotopy-cohomology}
For $\mathcal{Y}\in {}^\perp\SH(k)_{t\ge 1}$ and $A\in \SH(k)_{t\ge 0}$ the canonical morphism $A\to \Ht^{\mathrm{t}}_0(A)$ induces an isomorphism
\[
[\mathcal{Y},A]\xrightarrow{\simeq} [\mathcal{Y},\Ht^{\mathrm{t}}_0(A)].
\] 
\end{lemma}
\begin{proof}
Consider the exact sequence
\[
[\mathcal{Y},A_{t\ge 1}] \to [\mathcal{Y},A] \to [\mathcal{Y},\Ht^{\mathrm{t}}_0(A)] \to [\mathcal{Y},A_{t\ge 1}[1]]
\]
associated to the triangle $A_{t\ge 1}\to A\to \Ht^{\mathrm{t}}_0(A)\to A_{t\ge 1}[1]$. Both the side terms are zero, whence the claim.
\end{proof}

\begin{corollary} \label{cor:coh_coh}
Let $X$ be a smooth projective variety and $U\subset X$ be an open subvariety with $Z=X-U$ being smooth. Then for $A\in \SH(k)_{t\ge 0}$ the canonical morphism $A\to \Ht^{\mathrm{t}}_0(A)$ induces isomorphisms 
\[
[(\Sigma_\T^\infty X_+)^\vee,A]\xrightarrow{\simeq} [(\Sigma_\T^\infty X_+)^\vee, \Ht^{\mathrm{t}}_0(A)],\quad
[(\Sigma_\T^\infty U_+)^\vee,A]\xrightarrow{\simeq} [(\Sigma_\T^\infty U_+)^\vee, \Ht^{\mathrm{t}}_0(A)].
\]
\end{corollary}
\begin{proof}
$\Sigma_\T^\infty X_+$ and $\Sigma_\T^\infty U_+$ are strongly dualizable by Theorem~\ref{thm:Atiyah_duality} and Remark~\ref{rem:U_dualizable}. Moreover, $\Sigma_\T^\infty X_+,\Sigma_\T^\infty U_+\in \SH(k)_{t\ge 0}$ by \cite[Theorem~4.2.10]{Mor04}. The claim follows from Lemmas~\ref{lm:dual_coh_dim} and~\ref{lm:cohomotopy-cohomology}.
\end{proof}

Recall the following cornerstone computation due to F.~Morel. 

\begin{theorem}[{\cite[Theorems~4.2.10 and~6.4.1]{Mor04} and \cite[Corollary~6.43]{Mor12}}] \label{thm:Morel} The sphere spectrum $\SSp$ is $(-1)$-connected (i.e. $\SSp\in \SH(k)_{t\ge 0}$) and $\underline{\pi}_0^{\A^{1}}(\SSp)_n\cong \K^{\MW}_n$.
\end{theorem}

\begin{remark}
There is another approach to the above computation given in \cite[Theorem~9.6]{N14} that is based on the theory of presheaves with framed transfers \cite{GP14,GP15}.
\end{remark}

\begin{theorem}[{cf. \cite[Theorem~4.3.1]{AH11}}] \label{thm:hom_variety}
Let $X$ be a smooth projective variety of dimension $d$. Then there exist canonical isomorphisms
\[
[\SSp,\Sigma_\T^\infty X_+ (l)]\cong \coh^d(\RS^X(X;\omega_X;\coh^t_0(\SSp)(d+l))) \cong\coh_{\Nis}^d(X;\K_{d+l}^{\MW}\otimes \omega_X)\cong \coh_{\Zar}^d(X;\K_{d+l}^{\MW}\otimes \omega_X).
\]
Here $\omega_X$ is the canonical line bundle and the complex $\RS^X(X;\omega_X;\coh^t_0(\SSp)(d+l))$ looks as follows:
\begin{multline*}
\bigoplus_{x\in X^{(0)}} \K^{\MW}_{d+l}(x; (\omega_{X})|_x)  \to \bigoplus_{x\in X^{(1)}} \K^{\MW}_{d+l-1}(x; \Lambda_x^X\otimes (\omega_{X})|_x)  \to \hdots \\
\hdots \to \bigoplus_{x\in X^{(d-1)}} \K^{\MW}_{l+1} (x; \Lambda_x^X\otimes (\omega_{X})|_x)  \to \bigoplus_{x\in X^{(d)}} \K^{\MW}_{l} (x).
\end{multline*}
Let $i\colon \Spec k\to X$ be a rational point and $l=0$. Then, under the above isomorphism, the morphism $\Sigma_\T^\infty i_+\in [\SSp,\Sigma_\T^\infty X_+]$ corresponds to the cohomology class of $\langle 1 \rangle_x\in \mathcal{RS}_d^X(X;\omega_X;\coh^t_0(\SSp)(d))$, where $\langle 1\rangle_x\in \mathrm{GW}(k)=\K^{\MW}_{0}(\Spec k)$ belongs to the summand corresponding to $x=i(\Spec k)\in X^{(d)}$.
\end{theorem}
\begin{proof}
Choose a Zariski locally trivial $\A^s$-bundle $\rho\colon \widetilde{X}\to X$ and an isomorphism $\rho^*T_X\oplus E\cong \triv^N_{\widetilde{X}}$ as in Definition~\ref{def:Thom_stable_normal}. We have the isomorphisms
\[
[\SSp,\Sigma_\T^\infty X_+(l)]\cong [\Th(-T_X),\SSp(l)]\cong [\Th(-T_X),\coh^t_0(\SSp)(l)]\cong [\Th(E),\coh^t_0(\SSp)(N+l)[N]].
\]
The first one is given by Theorem~\ref{thm:Atiyah_duality}, the second one follows from Corollary~\ref{cor:coh_coh} and the last one is given by suspension.

The first isomorphism of the theorem follows from Lemma~\ref{lm:fibration_RS} and Theorem~\ref{thm:Morel}. The identification of the cohomology of the Rost-Schmid complex and the sheaf cohomology follows from Theorem~\ref{thm:Hom_coh}. The identification of $\Sigma_\T^\infty i_+$ follows from Lemma~\ref{lm:dual_point}.
\end{proof}

\begin{theorem} \label{thm:hom_curve}
Let $C_o$ be a smooth curve with a smooth compactification $C$ and $D=C-C_o$. Then there are canonical isomorphisms
\begin{multline*}
[\SSp,\Sigma_\T^\infty (C_o)_+(l)]\cong \coh^1(\RS^C(C,D;\omega_C;\coh^t_0(\SSp)(l+1))) \cong\\
\cong \coh_{\Nis}^1(C;(\K_{l+1}^{\MW}\otimes \omega_C)_{C,D})\cong \coh_{\Zar}^1(C;(\K_{l+1}^{\MW}\otimes \omega_C)_{C,D}).
\end{multline*}
Here 
\begin{itemize}
	\item
	$\omega_C$ is the canonical line bundle,
	\item
	$(\K_{l+1}^{\MW}\otimes \omega_C)_{C,D}=\ker\left[\K_{l+1}^{\MW}\otimes \omega_C\xrightarrow{i_D^*} (i_D)_* (\K_{l+1}^{\MW} \otimes i_D^*\omega_C) \right]$ with $i_D$ being the closed embedding $i_D\colon D\to C$, 
	\item
	the complex $\RS^C(C,D;\omega_C;\coh^t_0(\SSp)(l+1))$ consists of two terms
	\[
	\K^{\MW}_{l+1}(\Spec \struct_{C,D}; (\omega_{C})|_{\Spec \struct_{C,D}}) \to \bigoplus_{x\in C_o^{(1)}} \K^{\MW}_{l} (x) \oplus \bigoplus_{x\in D} \K^{\MW}_{l+1} (x; (\omega_C)|_x).
	\]
\end{itemize}
Let $i\colon \Spec k\to C_o$ be a rational point and $l=0$. Then under the above isomorphism $\Sigma_\T^\infty i_+\in [\SSp,\Sigma_\T^\infty (C_o)_+]$ corresponds to the cohomology class of $\langle 1 \rangle_x\in \mathcal{RS}_1^C(C,D;\omega_C;\coh^t_0(\SSp)(1))$, where $\langle 1\rangle_x\in \mathrm{GW}(k)=\K^{\MW}_0(\Spec k)$ belongs to the summand corresponding to $x=i(\Spec k)$.
\end{theorem}

\begin{proof}
Choose a Zariski locally trivial $\A^s$-bundle $\rho\colon \widetilde{C}\to C$ and an isomorphism $\rho^*T_{C} \oplus E \cong \triv^N_{\widetilde{C}}$ as in Definition~\ref{def:Thom_stable_normal}. 

Consider the triangle
\[
\Sigma_\T^\infty (C_o)_+ \to \Sigma_\T^\infty C_+ \to \Sigma_\T^\infty (C/C_o) \to \Sigma_\T^\infty (C_o)_+[1].
\]
Identifying $C/C_o\cong \Th(N_{D/C}) =((T_C)|_D)$ by the homotopy purity theorem \cite[Theorem 2.23]{MV99}, applying $\iHom_{\SH(k)}(-,\SSp)$ and rotating we obtain a triangle
\[
(\Sigma_\T^\infty \Th((T_C)|_D))^\vee \to (\Sigma_\T^\infty C_+)^\vee \to (\Sigma_\T^\infty (C_o)_+)^\vee \to (\Sigma_\T^\infty \Th((T_C)|_D))^\vee[1].
\]
Suspending with $\Sigma_\T^N$ and applying the isomorphism $(\Sigma_\T^\infty C_+)^\vee \cong \Th(-T_{C})$ from Theorem~\ref{thm:Atiyah_duality} we obtain a triangle
\[
\Sigma_\T^\infty \Th(E|_{\widetilde{D}}) \to \Sigma_\T^\infty \Th(E) \to \Sigma_\T^N (\Sigma_\T^\infty (C_o)_+)^\vee \to \Sigma_\T^\infty \Th(E|_{\widetilde{D}})[1].
\]
Here $\widetilde{D}=\rho^{-1}(D)$ and the identification $(\Sigma_\T^\infty \Th((T_C)|_D))^\vee\cong \Sigma_\T^\infty \Th(E|_{\widetilde{D}})$ is given by the proof of Theorem~\ref{thm:Atiyah_duality}, i.e. by the proofs of \cite[Theorem~A.1]{H05} and \cite[{Corollary~6.13}]{H15}. Moreover, it follows from loc. cit. that the first map in the triangle is induced by the embedding $D\to C$. Thus 
\[
C_o^\vee=(\Sigma_\T^\infty (C_o)_+)^\vee \cong \Sigma_\T^{-N} \Sigma_\T^\infty (\Th(E)/\Th(E|_{\widetilde{D}})).
\]

We have the isomorphisms
\[
[\SSp,\Sigma_\T^\infty (C_o)_+(l)]\cong [C_o^\vee,\SSp(l)]\cong [C_o^\vee,\coh^t_0(\SSp)(l)]\cong [\Sigma_\T^\infty \Th(E)/\Th(E|_{\widetilde{D}}),\coh^t_0(\SSp)(N+l)[N]].
\]
The first one is given by the strong dualizability of $\Sigma_\T^\infty (C_o)_+$, i.e. by Theorem~\ref{thm:Atiyah_duality} and Remark~\ref{rem:U_dualizable}, the second one follows from Corollary~\ref{cor:coh_coh} and the last one is given by the above considerations.

The first isomorphism of the theorem follows from Lemma~\ref{lm:fibration_RS_CD} and Theorem~\ref{thm:Morel}. The identification of the cohomology of the Rost-Schmid complex and the sheaf cohomology follows from Theorem~\ref{thm:Hom_coh_CD} and an observation that for a closed point $x\in C$ the restriction homomorphism 
\[
\K^\MW_{N+l}(\Spec \struct_{C,x}) \to \K^\MW_{N+l}(x) 
\]
is clearly surjective. The identification of $\Sigma_\T^\infty i_+$ follows from Lemma~\ref{lm:dual_point}.
\end{proof}

\begin{remark} \label{rem:hom_curve}
It follows from the surjectivity of the restriction morphism 
\[
\K^{\MW}_{1}(\Spec \struct_{C,D})\to \K_1^{\MW}(D)
\]
that $[\SSp,\Sigma_\T^\infty (C_o)_+]\cong \coh^1(\RS^C(C,D;\omega_C;\coh^t_0(\SSp)(1)))$ is canonically isomorphic to the cokernel
\[
(\K^{\MW}_{1})_{C,D}(\Spec \struct_{C,D}; (\omega_{C})|_{\Spec \struct_{C,D}}) \xrightarrow{\partial} \bigoplus_{x\in C_o^{(1)}} \mathrm{GW} (k(x)).
\]
This description is similar to the one given in \cite[Lemma~2.3]{SV96} for the relative Picard group.
\end{remark}

\appendix

\section{Some lemmas about presheaves with framed transfers}
In this section we follow the exposition of \cite{GP14} and \cite{GP15}.

\begin{definition}
Let $X$ and $Y$ be smooth varieties. An \textit{explicit framed correspondence} $\Phi$ of level $m$ from $X$ to $Y$ consists of the following data:
\begin{enumerate}
	\item
	a closed subset $S$ in $\A^m_X$ which is finite over $X$;
	\item
	an etale neighborhood $p\colon U\to \A^m_X$ of $S$;
	\item
	a collection of regular functions $\phi=(\phi_1,\phi_2,\hdots,\phi_m)$ on $U$ such that $S=\{\phi=0\}$;
	\item
	a regular morphism $U\to Y$.
\end{enumerate}
Two explicit framed correspondences $(S,U,\phi,g)$ and $(S',U',\phi',g')$ of level $m$ are said to be equivalent if $S=S'$ and there exists an etale neighborhood $V$ of $S$ in $U\times_{\A^m_X} U'$ such that $g\circ \pi_U=g'\circ \pi_U'$ and $\phi\circ \pi_U=\phi'\circ \pi_{U'}$ for the respective projections $\pi_U\colon V\to U$ and $\pi_{U'}\colon V\to U'$. The set of level $m$ framed correspondences (i.e. explicit framed correspondences up to the above equivalence) is denoted $\Fr_m(X,Y)$. Note that $\Fr_0(X,Y)=\mathrm{Mor}_{\Sm_{k,\bullet}}(X_+,Y_+)$ is the set of pointed regular morphisms between $X$ and $Y$ pointed externally. Put 
\[
\ZF_m(X,Y)=\Z[\Fr_m(X,Y)]/A,
\]
where $\Z[\Fr_m(X,Y)]$ is the free abelian group on the set of level $m$ framed correspondences and $A$ is the subgroup generated by the elements
\[
(S\sqcup S', U, \phi, g) - (S, U-S', \phi|_{U-S'}, g|_{U-S'}) - (S', U-S, \phi|_{U-S}, g|_{U-S}),
\]
and denote $\ZF_*(X,Y)=\bigoplus_{m\ge 0} \ZF_m(X,Y)$.

Let $X,Y$ and $Z$ be smooth varieties and let $\Phi=(S,U,\phi,g)\in \Fr_m(X,Y)$ and $\Psi=(T,W,\psi,h)\in \Fr_l(Y,Z)$ be explicit correspondences. Then we compose them in the following way (see the details in \cite{GP14}):
\[
\Psi \circ \Phi = (S\times_Y T,U\times_Y W, (\phi\circ \pi_U, \psi\circ \pi_W), h\circ \pi_W).
\]
One can show that this rule induces a composition $\ZF_m(X,Y)\times \ZF_l(Y,Z)\to \ZF_{m+l}(X,Z)$ that is associative. The \textit{category of linear framed correspondences} $\ZF_*(k)$ has objects those of $\Sm_k$ and morphisms are the abelian groups $\ZF_*(X,Y)$. A \textit{linear presheaf with framed transfers} is a contravariant functor from the category $\ZF_*(k)$ to the category of abelian groups.

We say that a linear presheaf with framed transfers $\mathcal{F}$ is \textit{homotopy invariant} if the canonical projection induces an isomorphism $\mathcal{F}(X)\cong \mathcal{F}(X\times \A^1)$ for every $X\in\Sm_k$. 

For every $X\in \Sm_k$ denote 
\[
\sigma_X=(\{0\}\times X,\A^1\times X, \pi_{\A^1},\pi_X)\in \Fr_1(X,X)
\]
the suspension morphism. Here $\pi_{\A^1}$ and $\pi_X$ are the respective projections. We say that a linear presheaf with framed transfers $\mathcal{F}$ is \textit{quasi-stable} if $\mathcal{F}(\sigma_X)$ is an isomorphism for every $X\in\Sm_k$.
\end{definition}

\begin{definition}\label{def:homotopy_transfers}
Let $X$ and $Y$ be smooth varieties. An explicit framed correspondence $\Phi=(S,U,\phi,g)$ of level $m$ from $X$ to $Y$ gives rise to a morphism of Nisnevich sheaves 
\[
\theta(\Phi)\colon (\Proj^1,\infty)^{\wedge m} \wedge X_+ \to \T^{\wedge m} \wedge Y_+
\]
in the following way. Consider commutative diagram
\[
\xymatrix{
	U - S \ar[r] \ar[d] & U \ar@/^1.5pc/[ddr]^{(\phi,g)} \ar[d]^p  & \\
	(\Proj^1)^{\times m} \times X - S \ar[r]^i \ar@/_1.5pc/[drr]^{q} & (\Proj^1)^{\times m} \times X & \\
	& & \T^{\wedge m}\wedge Y_+ \cong \A^n_Y/((\A^m-0) \times Y)
}
\]
Here the square is cartesian, $p$ is given by the composition $U\to \A^m_X\to (\Proj^1)^{\times m}\times X$ for the standard embedding $\A^1=\Proj^1-\infty\subset \Proj^1$, $i$ is the open embedding and $q$ is the constant morphism that maps $(\Proj^1)^{\times m}_X - S$ to the distinguished point. The square is a Nisnevich cover, thus we have a morphism of Nisnevich sheaves
\[
((\Proj^1)^{\times m} \times X)_+ \to \T^{\wedge m} \wedge Y_+
\]
that induces a morphism
\[
\theta(\Phi)\colon (\Proj^1,\infty)^{\wedge m} \wedge X_+ \to \T^{\wedge m} \wedge Y_+.
\]
One can show \cite[Lemma~5.2]{GP14} that this rule gives a natural bijection 
\[
\theta\colon \Fr_m(X,Y)\xrightarrow{\simeq} \operatorname{Map}_{{\mathrm{Shv_\Nis}}_\bullet}((\Proj^1,\infty)^{\wedge m} \wedge X_+,\T^{\wedge m} \wedge Y_+).
\]

For $A\in\SH(k)$ and $\Phi\in \Fr_m(X,Y)$ let
\[
\Phi^*\colon \pi^{\A^1}_i(A)_n(Y) \to \pi^{\A^1}_i(A)_n(X)
\]
be given by the composition
\begin{multline*}
\pi^{\A^1}_i(A)_n(Y)=[Y[i],A(n)]\cong [\T^{\wedge m}\wedge Y_+[i],\T^{\wedge m}\wedge A(n)] \xrightarrow{\theta(\Phi)^*} \\
\xrightarrow{\theta(\Phi)^*} [(\Proj^1,\infty)^{\wedge m}\wedge X_+[i],\T^{\wedge m}\wedge A(n)] \cong [T^{\wedge m}\wedge X_+[i],\T^{\wedge m}\wedge A(n)] \cong 
\\ \cong [X[i],A(n)] = \pi^{\A^1}_i(A)_n(X).
\end{multline*}
Here we used the canonical isomorphism $\T\cong (\PP^1,\infty)$ and the suspension isomorphism $\Sigma^m_{\T}$. One can check that this rule endows the presheaf of homotopy groups $\pi^{\A^1}_i(A)_n(-)$ with the structure of a \textit{homotopy invariant quasi-stable linear presheaf with framed transfers} (see the discussion of radditive presheaves in \cite[page~2]{GP15}).
\end{definition}

\begin{lemma}[{cf.~\cite[Theorem~22.2]{MVW06}}] \label{lm:ZarNis}
Let $\mathcal{F}$ be a homotopy invariant quasi-stable linear presheaf with framed transfers, e.g. $\mathcal{F}=\pi^{\A^1}_i(A)_n$ for $A\in\SH(k)$. Then the associated Zariski sheaf $\underline{\mathcal{F}}$ is a Nisnevich sheaf.
\end{lemma}
\begin{proof}
By \cite[{Theorem~2.1}]{GP15} the associated Nisnevich sheaf $\underline{\mathcal{F}}^{\Nis}$ has a canonical structure of a linear presheaf with framed transfers that is homotopy invariant and quasi-stable. Moreover, the associated morphism
\[
\mathcal{F} \to \underline{\mathcal{F}}^{\Nis}
\]
is a morphism of linear presheaves with framed transfers. It is sufficient to check that for a smooth variety $X$ and a point $x\in X$ the corresponding morphism $\mathcal{F}(\Spec \struct_{X,x}) \to \underline{\mathcal{F}}^{\Nis}(\Spec \struct_{X,x})$ has trivial kernel and cokernel. Put $F=k(X)$ for the fraction field of $\struct_{X,x}$ and consider the following diagram.
\[
\xymatrix{
	\mathcal{F}(\Spec \struct_{X,x}) \ar[r] \ar[d] & \underline{\mathcal{F}}^{\Nis}(\Spec \struct_{X,x})  \ar[r] \ar[d] & (\underline{\mathcal{F}}^{\Nis}/\mathcal{F}) (\Spec \struct_{X,x}) \ar[d] \\
	\mathcal{F}(\Spec F) \ar[r]^{\simeq} & \underline{\mathcal{F}}^{\Nis}(\Spec F)  \ar[r] & (\underline{\mathcal{F}}^{\Nis}/\mathcal{F})(\Spec F)
}
\]
Here all the vertical morphisms are injective by \cite[Theorem~2.15(3)]{GP15}. Moreover, $F$ is a field thus $\mathcal{F}(\Spec F) \cong \underline{\mathcal{F}}^{\Nis}(\Spec F)$ and $(\underline{\mathcal{F}}^{\Nis}/\mathcal{F})(\Spec F)=0$. The claim follows.
\end{proof}

\begin{lemma} \label{lm:coh_semiloc}
Let $\mathcal{F}$ be a homotopy invariant quasi-stable Nisnevich sheaf with framed transfers, e.g. $\mathcal{F}=\underline{\pi}^{\A^1}_i(A)_n$ for $A\in\SH(k)$. Then for a smooth variety $X$ and a finite collection of closed points $D\subset X^{(d)}$ we have
\[
\coh^1_\Nis(\Spec \struct_{X,D},\mathcal{F})=0.
\]
\end{lemma}
\begin{proof}
By \cite[Corollary~14.4 and Theorem~14.14]{GP15} presheaf $\coh^1_\Nis(-,\mathcal{F})$ has a canonical structure of a linear presheaf with framed transfers that is homotopy invariant and quasi-stable. Then 
\[
\coh^1_\Nis(\Spec \struct_{X,D},\mathcal{F}) \to \coh^1_\Nis(\Spec k(X),\mathcal{F}) = 0
\]
is injective by \cite[Theorem~2.15(3)]{GP15} (literally the same proof as in loc.cit. works for a semilocal scheme).
\end{proof}

\section*{Acknowledgments}
I would like to thank the participants of the $\A^1$-homotopy theory seminar in Chebyshev Laboratory and especially Mikhail Bondarko, Ivan Panin and Vladimir Sosnilo for valuable comments and suggestions. The research is supported by the Russian Science Foundation grant №14-21-00035.

\end{document}